\documentclass[11pt,oneside,reqno]{amsart}
\usepackage{amsmath}
\usepackage{amsfonts}
\usepackage{amssymb}
\usepackage{graphicx}
\usepackage{enumerate}
\usepackage{enumitem} 
\usepackage{subcaption}
\usepackage{dirtytalk}
\usepackage[mathscr]{eucal}
\newtheorem{theorem}{Theorem}
\theoremstyle{plain}

\newtheorem{corollary}{Corollary}

\newtheorem{fact}{Fact}

\newtheorem{lemma}{Lemma}

\newtheorem{proposition}{Proposition}
\theoremstyle{remark}
\newtheorem{remark}{Remark}

\theoremstyle{plain}

\theoremstyle{definition}
\newtheorem{definition}{Definition}

\theoremstyle{definition}

\newtheorem*{cexmp}{Counter-examples}
\usepackage[height=205mm,width=142mm]{geometry}   
\usepackage{mathtools}

\begin{document}
\title[Packing of permutations into Latin squares
]{
Packing of permutations into Latin squares}

\author{Stephan Foldes}
\address{Stephan Foldes\newline%
\indent University of Miskolc,  \newline%
\indent 3515 Miskolc-Egyetemvaros, Hungary}
\email {foldes.istvan@uni-miskolc.hu}%

\author{Andr\'as Kaszanyitzky}
\address{Andr\'as Kaszanyitzky\newline%
\indent E\"otv\"os Lor\'and University \newline%
\indent Faculty of Science / Museum of Mathematics\newline%
\indent  1117 Budapest, P\'azm\'any P\'eter s\'et\'any 1/a, Hungary }
\email{kaszi75@gmail.com}%

\author{L\'aszl\'o Major}
\address{L\'aszl\'o Major\newline%
\indent University of Turku   \newline%
\indent Faculty of Science and Engineering \newline%
\indent  20500 Turku, Vesilinnantie 5, Finland }
\email{laszlo.major@utu.fi}%

\hspace{-4mm} \date{21 December 2019}

\begin{abstract}
For every positive integer $n$ greater than $4$ there is a set of  Latin squares of order $n$ such that every permutation of the numbers $1,\ldots,n$ appears exactly once as a row, a column, a reverse row or a reverse column of one of the given Latin squares.  If $n$ is greater than $4$ and not of the form $p$ or $2p$ for some prime number $p$ congruent to $3$ modulo $4$, then there always exists a Latin square of order $n$ in which the rows, columns, reverse rows and reverse columns are all distinct permutations of $1,\ldots,n$, and which constitute a permutation group of order $4n$. If $n$ is prime congruent to $1$ modulo $4$, then a set of $(n-1)/4$ mutually orthogonal Latin squares of order $n$ can also be constructed by a classical method of linear algebra in such a way, that the rows, columns, reverse rows and reverse columns are all distinct and constitute a permutation group of order $n(n-1)$.   

 \end{abstract}

\maketitle

\section{Introduction}

A \textit{permutation} is a bijective map  $f$  from a finite, $n$-element set (the domain of the permutation) to itself. When the domain is fixed, or it is an arbitrary $n$-element set, the group of all permutations on that domain is denoted by $S_n$ (\textit{full symmetric group}). If the elements of the domain are enumerated in some well defined order as $z_1,\ldots,z_n$, then the sequence  $f(z_1),\ldots,f(z_n)$  is called the \textit{sequence representation} of the permutation $f$.  This sequence representation then fully determines the permutation, and the term \say{permutation} may mean this sequence itself or the bijective map that it represents.

In every $n$-by-$n$ square matrix $L$ there are sequences of length $n$ occurring as rows, columns, reverse rows or reverse columns. These at most $4n$ sequences will be called \textit{lines} of the matrix. 

Without loss of generality we shall throughout this paper assume that matrix entries and the elements of the underlying sets of permutations belong to some non-trivial ring with a zero and a unit element $1$. This will allow in particular the multiplication of any square matrix of order $n$ with the 
 \emph{reversal matrix} (also called the \emph{exchange matrix}) of order $n$ given by $J(i,j)= \delta_{i, n+1-j}$ (where $\delta_{i,j} $ is the Kronecker delta). 
  Thus the reverse columns of $L$ are the rows of the conjugate $JLJ$
and the reverse rows are the rows of the conjugate transpose $JL^TJ$, while the columns are just the rows of the transpose $L^T$.

The four matrix operators mapping $L$ to $L$, $L^T$, $JLJ$, and $JL^TJ$, respectively, form a group under composition. The matrix $L$ is \textit{symmetric} if $L=L^T,$ \textit{centrosymmetric} if $L=JLJ$  and \textit{Hankel symmetric} if  $L = JL^TJ$.  None of these properties implies the others, but any two of them imply the third: such matrices have been studied e.g. in \cite{cent0, cent1, cent2, cent10}.

The opposite of this situation is when the $4n$ lines of the matrix $L$ are all distinct, in this paper we call such a matrix \textit{strongly asymmetric}. In this paper we are interested in matrices that are actually Latin squares. For the general theory of Latin squares see Keedwell and D\'enes \cite{kur}.

Given a finite, $n$-element set $R$, a set $S$ of $m$  Latin squares of order $n$ with entries in $R$ (meaning that each line of each member of $S$ represents a permutation of $R$) is called a \textit{packing} if 
\begin{enumerate}[label=(\roman*)]
\item
 all members of $S$ are strongly asymmetric,
\item no line of any member of $S$ appears as a line in any other member of $S$. 
\end{enumerate}

Equivalently, $S$ is a packing if the number of distinct sequences appearing as lines in the members of $S$ is $4nm$. In that case the $4nm$ permutations represented by these $4nm$ sequences are said to be \textit{packed into}  the set $S$ of Latin squares.

\begin{proposition}\label{els}
If some subgroup $G$ of the full symmetric group $S_n$ of all permutations is packed into a set $S$ of strongly asymmetric Latin squares, then every subgroup $H$ of $S_n$ containing $G$ is also packed into some set of strongly asymmetric
Latin squares.
\end{proposition}
\begin{proof}
Let $Q$ be a set of $|H|/4n$ distinct representatives of the (left) cosets of $G$ in $H$. Apply each representative to each matrix $M$ in S elementwise, i.e. for each $q$ in $Q$ and each $M$ form the matrix $qM$ given by $(qM)(i,j) = q(M(i,j))$.
\end{proof}

\section{Packing into Latin squares of odd order}

\begin{fact}
Let $C$ and $R$ be subgroups of a group $G$. For any element $g\in G$ the following conditions are equivalent:
\begin{enumerate}[label=(\roman*)]
\item the map from the Cartesian product set $C\times R$ to the double coset  $CgR$  that sends $(c,r)$ to $cgr$ is bijective,

\item 
the conjugate subgroup $gRg^{-1}$  intersects $C$ trivially.
\end{enumerate}
\end{fact}

\begin{fact}
Let $C$ and $R$ be subgroups of a finite group $G$. The following conditions are equivalent:
\begin{enumerate}[label=(\roman*)]
\item all the double cosets have cardinality equal to Card$(C) \times$Card$(R)$,

\item 
every conjugate of $C$ meets every conjugate of $R$ trivially.
\end{enumerate}
\end{fact}

\begin{fact}\label{fact}

Let $C$ and $R$ be two subgroups of a finite group $G$ whose orders are relatively prime. Then all the double cosets have cardinality equal to Card$(C) \times$Card$(R)$. 
\end{fact}
The following theorem will be proved by an appropriate application of the \say{addition square} construction appearing in Gilbert \cite{gil}.

\begin{theorem}
If $n$ is a positive odd number at least 5, then the full symmetric group $S_n$ can be packed into some set of strongly asymmetric Latin squares.
\end{theorem}

\begin{proof}
Fact \ref{fact} applies in particular if $G$ is the full symmetric group of all permutations of the $n$-element set $\{1,\ldots,n\}$, for odd $n\geq 5$, $C$ is the subgroup generated by the circular shift permutation whose sequence representation is $(2,\ldots,n,1)$, and $R$ is the two-element subgroup containing the reversal permutation whose sequence representation is $(n,\ldots,1)$. 

All double cosets $CgR$ have cardinality $2n$. The number $(n-1)!/2$ of double cosets is even. 
Arrange the double cosets in matched pairs.
From each matched pair $(A,B)$ of double cosets choose representatives $p,q$ from $A,B,$ respectively.
Construct the matrix $M$ given by  $M(i,j) = p(i) + q(j)$.

The columns and reverse columns of $M$ are precisely the members of the double coset $A$, and the rows and reverse rows of $M$ are the members of the double coset $B$.
In the set of $(n-1)!/4$ matrices so constructed every permutation occurs exactly once as a row, column, reverse row or reverse column.
\end{proof}

\begin{remark}
The above construction does not work for even $n$, because the circular shift applied $n/2$ times is then a conjugate of the reversal (has similar cyclic structure).
\end{remark}

\section{Packing into Latin squares of even order}

\begin{definition} An $m$-by-$m$ matrix (possibly Latin square) is called a \textit{double occurrence matrix} if every sequence appearing as a {line} (row, column, reverse row, or reverse column) appears exactly twice, and every sequence appearing as a row of index $i$ ($i$-th row) also appears as the column of the same index $i$. A set of double occurrence matrices is called a \textit{double occurrence matrix set} if no sequence appears as a line in more than one of the matrices. The set of permutations whose sequence representations appear as lines in a double occurrence matrix or matrix set is called the \textit{set covered} by the matrix or matrix set.
\end{definition}
\begin{proposition}\label{hus}  For every $m>2$ there is a double occurrence matrix set covering all permutations of $1,\ldots,m$. \end{proposition}
\begin{proof}
First we observe that for every $m > 2$, there is an $m$-by-$m$  double occurrence matrix such that the lines appearing in the matrix constitute the sequence representations of a permutation group on the numbers $1,\ldots,m$. Indeed, such a matrix can be defined  by letting the element of the matrix in row $i$ and column $j$ be the number $i+j$, with addition defined 
modulo $m$.

Secondly, let $M$ be any $m$-by-$m$ double occurrence matrix whose lines constitute a permutation group $G$. Let $R$ be any complete system of representatives of the left cosets of $G$ in the full symmetric group of all permutations of $1,\ldots,m$. Applying the various permutations $r \in R$ (viewed as functions) entrywise to $M$ we obtain the matrix set required.
\end{proof}

\begin{remark}
The matrix set whose existence is stated in Proposition \ref{hus} necessarily consists of $m!/2m = (m-1)!/2$ matrices.
\end{remark}

For $k=1,...,m$ let $B_{k_0}$ (respectively $B_{k_1}$) denote the $2$-by-$2$ matrix with first row $2k-1, 2k$, second row $2k, 2k-1$ (respectively first row $2k,2k-1$, second row $2k-1,2k$).

For an $m$-by-$m$ Latin square $L$ and each $m$-by-$m$ Boolean matrix $A$ ($0-1$ matrix), let the $2m$-by-$2m$ \textit{composite matrix}  $L*A$  be defined by replacing each entry of $L$ in row $i$ and column $j$ whose value is $k$ by the $2$-by-$2$ block $B_{k_{A(i,j)}}$ (for an example see Figure \ref{fig}).

\begin{figure}[ht!]
  
    \begin{subtable}{.2\linewidth}
      \centering
      
      L= \begin{tabular}{ |c|c|c| } 
 \hline
 1 & 3 & 2 \\
 \hline
 2 & 1 & 3 \\ 
 \hline
 3 & 2 & 1 \\ 
 
 \hline
\end{tabular}
    \end{subtable}%
    \begin{subtable}{.2\linewidth}
      \centering
       
         A= \begin{tabular}{ |c|c|c| } 
 \hline
 
 1 & 0 & 1 \\
 \hline
 0 & 0 & 1 \\ 
 \hline
 0 & 1 & 0 \\

 \hline
\end{tabular}
    \end{subtable} 
     \begin{subtable}{.4\linewidth}
      \centering
      $\longrightarrow L*A=$
      \begin{tabular}{ |c|c||c|c||c|c| } 
 \hline
 2 & 1 & 5 & 6 & 4 & 3\\\hline
 1 & 2 & 6 & 5 & 3 & 4\\\hline
 \hline
 3 & 4 & 1 & 2 & 6 & 5\\\hline
 4 & 3 & 2 & 1 & 5 & 6\\\hline
 \hline
 5 & 6 & 4 & 3 & 1 & 2\\\hline
 6 & 5 & 3 & 4 & 2 & 1\\\hline

\end{tabular}
    \end{subtable} 
     \caption{}
     \label{fig}
\end{figure}

\begin{remark}\label{kas}
Both $L$ and $A$ are encoded into the composite $L*A$ without loss of information: $A(i,j)$ is $0$ or $1$ according to whether $(L*A) (2i-1, 2j-1)$ is smaller or larger than $(L*A) (2i-1, 2j)$, and $L(i,j)$ is half the larger of these two entries of $L*A$.
\end{remark}
\begin{proposition}\label{fos} For every even number $2m>4$, there is a set $S$ of Latin squares of order $2m$, such that every sequence representing a permutation preserving the partition of the numbers $1,2,\ldots,2m-1,2m$ into pairs of consecutive numbers occurs exactly once as a line of some member of $S$. \end{proposition}

\begin{remark}
The set of permutations packed into the set $S$ of Latin squares according to this proposition constitutes a subgroup of order $m!2^m$ of the full symmetric group $S_{2m}$.
\end{remark}

\begin{proof}
Partition into pairs $\{v,w\}$ in any manner the set of $0-1$ vectors of length $m$ whose first component is $0$ (actually partitioning into pairs of any complement-free set of $2^{m-1}$ vectors will do). Order each pair $\{v,w\}$ arbitrarily into an ordered pair $(v,w)$. For each such ordered pair define the $m$-by-$m$  Boolean matrix $A_{vw}$ by $A_{vw}(i,j) = v(i) + w(j)$, with addition modulo $2$.

We obtain thus a set of $2^{m-2}$ Boolean matrices with the property that if a vector $v$ appears as the $i$-th row in one of the matrices, then the $i$-th column of that matrix is neither $v$ nor its Boolean complement.

Let $E$ be a double occurrence matrix set covering all permutations of $1,\ldots,m$, which exists according to Proposition \ref{hus} and consists of $(m-1)!/2$ matrices.

Let $S$ consist of all the composite matrices $(L*A)_{vw}$, where $L$ can be any member of $E$. In view of Remark \ref{kas}, the set $S$ consists of  $(m-1)! 2^{m-3}$  matrices.
The key to verifying that no vector can appear twice as a line anywhere in $S$, is to note that because of the double occurrence property of $E$ the only possible coincidences to worry about concern rows and columns of the same member matrix of $S$, but such coincidences are excluded due to the complement-free property of the set of Boolean vectors used to define the matrices $A_{vw}$.

The proof is concluded by observing that the number of (distinct) lines appearing in members of $S$ is $(m-1)! 2^{m-3} \cdot 8m = m! 2^m$, which is the number of the partition-preserving permutations specified in the statement of the proposition.
\end{proof}

\begin{theorem}
If $n$ is a positive even number at least $6$, then the full symmetric group $S_{n}$ can be packed into some set of strongly asymmetric Latin squares.
\end{theorem}

\begin{proof}
Proposition \ref{fos} packs a subgroup of the full symmetric group. Combining with Proposition \ref{els} yields the result.
\end{proof}

\section{Packing a permutation group into a single Latin square}

In this section we give a construction showing that if $n$ is either a composite odd integer or a prime congruent to $1 \mod 4$, then  the full symmetric group $S_n$ has a subgroup --- necessarily of order $4n$ --- that can be packed into a single strongly asymmetric Latin square. The construction will use basic properties of finite rings, where by a ring we always mean a commutative ring with unit element 1. 

In every finite cyclic group $\mathbb{Z}_m$ there is an element $u$ such that $x+x=u$ is only possible for at most one element $x$.  For this we can take $u=-1$. Taking direct products we can see that in every finite Abelian group there is an element $u$ such that $x+x=u$ is only possible for at most one element $x$.

Thus every finite ring $R$ has an element  $u$  such that  $2x=u$ has at most one solution $x$. With such an element  $u$  of $R$ being fixed, for each element $x$ of $R$ denote $u-x$ by $x'$. Mapping $x$ to $x'$ defines an involution on $R$, in this paper called \textit{reflection} (with respect to an element  $u$), which has at most one fixed point. It has a fixed point if and only if $2x=u$ has a solution $x$ (which happens if and only if the number of elements of $R$ is odd).

\begin{proposition}
 Given a reflection $x \mapsto x'$ of an $n$-element ring $R$ (with respect to an element   $u$), the elements of $R$ can be enumerated  $z_1,\ldots,z_n$ in such a way that the reverse sequence $z_n,\ldots,z_1$ coincides with the reflected sequence $z_1',\ldots,z'_{n}$.
\end{proposition}

\begin{proof}If the reflection has no fixed point, $n$ is even and the elements of $R$ can be partitioned into $n/2$ pairwise disjoint pairs so that the reflection exchanges the elements of each pair. Choose in any way one element from each pair and enumerate them in any order as $z_1,\ldots,z_{n/2}$. Then let $z_{n-i+1}= z_i'$ for $i=1,\ldots,n/2$.

If the reflection has a fixed point $y$, then the other elements of $R$ can be partitioned into  $(n-1)/2$  pairwise disjoint pairs so that the reflection exchanges the elements of each pair. Again, choose in any way one element from each pair and enumerate them in any order as $z_1,\ldots,z_{(n-1)/2}$.  Let $z_{(n+1)/2}= y$  and let  $z_{n-i+1}= z_i'$ for $i=1,\ldots,(n-1)/2$.
\end{proof} 

Fixing a reflection $x \mapsto x'$ of the ring $R$ (with respect to an element   $u$), and an enumeration $z_1,...,z_n$, we shall say that this  enumeration is  \textit{reflectable} if $z'_i=z_{n-i+1}$ for all $i$.

\begin{fact}
For every finite ring there exists a reflectable enumeration of its elements (with respect to an element   $u$).
\end{fact}

\begin{definition}
We shall say that in a finite ring $R$ a $4$-element subgroup $G$ of the multiplicative group of units is a \textit{quartet} if it consists of two distinct elements and their negatives. (Necessarily $1$ and $-1$ belong then to the quartet $G$.)

\end{definition}
\begin{lemma}\label{lem1}
If $q$ is a prime power congruent to $1$ mod $4$, then the finite field $GF(q)$ has a quartet.
\end{lemma}
\begin{proof}
The quartet consists of  $1$, $-1$ and the two square roots of $-1$.
\end{proof}

\begin{cexmp}
$GF(p)$ with a prime $p$ congruent to $3 \mod 4$, or $GF(q)$ with $q$ an odd power of a prime congruent to $3 \mod 4$, has no quartet.
\end{cexmp}
\vspace{-1mm}
\begin{proposition}\label{ekvi}
Let $n$ be an integer at least $5$. The following conditions are equivalent:
\begin{enumerate}[label=(\roman*)]
\item \label{i} there exists an $n$-element commutative ring $R$ with unit that has a quartet,
\item \label{ii} $n$ is not a prime congruent to $3$ mod $4$, and it is not $2$ times such a prime.
\end{enumerate}
\end{proposition}
\begin{proof}
If \ref{ii} does not hold but \ref{i} does, then the additive group of $R$ must be cyclic, therefore $R$ is isomorphic to $\mathbb{Z}_n$. Euler's phi function takes value $n-1$ on $n$, where $n$ is a prime, or the value $(n/2)-1$ if $n$ is $2$ times a prime, thus the order of the group of units of $\mathbb{Z}_n$ is not divisible by $4$ and therefore it cannot have any $4$-element subgroup, contradicting \ref{i}.

If \ref{ii} holds, we need to examine the following cases.  

\emph{Case 1:} if $n$ is composite odd, let $n=mq$ be a proper factorization. Obviously both $m$ and $q$ are greater then $2$. The direct product ring  $\mathbb{Z}_m \times \mathbb{Z}_q$  has a quartet, namely 
$(1,1),  (-1,-1),  (1,-1),  (-1,1)$. 

\emph{Case 2}:  if $n$ is prime congruent to $1$ mod $4$ then in $\mathbb{Z}_n=GF(n)$ the element $-1$ has a square root  $r$  and $\{1,-1,r,-r \}$ is a quartet (this is a special case of Lemma \ref{lem1}).

\emph{Case 3:}  if $n$ is divisible by $4$, then in $\mathbb{Z}_n$  the residues $1$, $-1$, $(n-2)/2$ and $(n+2)/2$ form a quartet.

\emph{Case 4:}  if $n$ is of the form $2m$ where $m$ is odd, then by applying Cases 1. and 2. we see that some $m$-element ring $A$ has a quartet $Q$. Then in the direct product ring $Z_2  \times A$ the set of elements of the form $(1,q)$, where $q$ is in $Q$, is a quartet.
\end{proof}

\begin{remark}
From the above proof it is clear that when $n$ satisfies the conditions of  Proposition \ref{ekvi}, then an $n$-element ring having a quartet exists that is either a ring of residue classes of $\mathbb{Z}$, or a direct product of two such rings, or a finite field. 
\\
There is no quartet in $\mathbb{Z}_9$ because the group of units of $\mathbb{Z}_9$ has order $6$.
\\
There is no quartet in $GF(27)$ because the $26$-element group of units of $GF(27)$ cannot have any $4$-element subgroup.
\\
There is no quartet in $\mathbb{Z}_{27}$ because the order of its group of units is $18$.
\\
There is a quartet in $GF(25)$ - consisting of the units in the prime subfield - but not in the direct product ring $\mathbb{Z}_5 \times \mathbb{Z}_5$.
\\
Because non-prime fields can be replaced by direct products of rings in constructing quartets (see Case 1 in the proof of Proposition \ref{ekvi}), rings of the form $\mathbb{Z}_k \times \mathbb{Z}_m$ suffice, where $k$ may be $1$, to construct rings with a quartet having the prescribed number of elements. 
\end{remark}

\begin{proposition}\label{pr1}
If $G$ is a quartet in an $n$-element ring $R$, with $c, d$ distinct elements of $G$ such that $c$ is not the negative of $d$, and if $z_1,\ldots,z_n$ is a reflectable enumeration of the elements of $R$ (with respect to an element $u$), then the $4n$ lines of the $n$-by-$n$ matrix $M $given by $M(i,j)=
c z_i + d z_j$ are all distinct. The matrix $M$ is a Latin square.
\end{proposition} 

\begin{proof}
 A column of $M$ is a sequence of the form
\begin{equation}\label{yks}
    c z_1 + b,\ldots, c z_n + b 
\end{equation}
for some element $b$ of $R$. The corresponding reverse column is the sequence $$c z_1' +b,\ldots, c z_n' + b$$ which is
$(-c) z_1 + (b+cu), \ldots, (-c) z_n + (b+cu)$,         
i.e. it is in the form
\begin{equation}\label{kaks}
    (-c) z_1 + a,\ldots, (-c) z_n + a    
\end{equation}
for some element  $a$  of $R$.  

As for distinct columns the corresponding elements  $b$  in (\ref{yks}) are distinct, the columns of the matrix $M$ are all distinct, and by a similar argument the rows of $M$ are distinct from each other too.

If a column (\ref{yks}) were to coincide with a reverse column (\ref{kaks}), then for all $i=1,\ldots,n$ we would have  $c z_i + b = (-c) z_i + a$.  In particular, taking $z_i = 0$, we would have to have $a = b$, and then $z_i = -z_i$ for all $i$ by the invertibility of $c$, implying in particular $1= -1$, which is impossible. Thus no column is identical with a reverse column.

Suppose that a column (\ref{yks}) were to coincide with a row. The row would be of the form
\begin{equation}\label{kol}
    e + d z_1 ,\ldots, e + d z_n 
\end{equation}
for some element  $e$  of $R$.  For $i=1,\ldots,n$ we would have  $c z_i + b = e + d z_i$. Setting  $z_i=0$ would imply $e=b$. Then setting $z_i=1$ would imply $c=d$, which is impossible. Thus no column is identical with a row.  

If a column (\ref{yks}) were to coincide with the reverse of a row of the form (\ref{kol}), then the reverse column (\ref{kaks}) would coincide with (\ref{kol}). Then for all we would have  
$(-c) z_i  + a = e + d z_i$, implying  $a = e$  and then $-c = d$, which is impossible by the choice of $c$ and $d$. Thus no column is identical with a reverse row.

Summarizing what we have seen so far: no column is identical with any other line. Similarly we can conclude that no row is identical with any other line. 
\end{proof}

Given a quartet $G$ in an $n$-element ring $R$, there are $4n$ distinct permutations of the elements of $R$ that are of the form $g i + b$, where $g$ is in $G$ and $b$ is in $R$. They form a subgroup $A$ of index $n!/4n$ in the group $S_n$ of all permutations of the elements of $R$. The $4n$ permutations in $A$ appear as the lines of the matrix $M$ defined in the statement of Proposition \ref{pr1}, where the definition of $M$ is based on the reflectable enumeration $z_1,\ldots,z_n$ of the elements of $R$. Thus in fact Proposition \ref{pr1} shows:
\begin{corollary}
If for the integer $n$ there is an $n$-element ring having a quartet, then the full symmetric group $S_n$ has a subgroup (necessarily of order $4n$) that can be packed into a single Latin square (whose  $4n$ lines  are all distinct).
\end{corollary}

\noindent Combining Propositions  \ref{ekvi} and  \ref{pr1} we have:

\begin{theorem}
If the integer $n > 4$ is not of the form $p$ or $2p$ with prime $p$ congruent to $3$ modulo $4$, then there exists a strongly asymmetric Latin square of order $n$ the $4n$ distinct lines of which form a permutation group (subgroup of $S_n$).
\end{theorem}

When $n$ is a prime congruent to $3$ modulo $4$, then by the non-existence of subgroups of order $4n$ in $S_n$ (see Proposition \ref{pr10} in the Appendix), obviously no permutation group can be packed into a single Latin square of order $n$. Finally, if $n$ is of the form $2p$ with $p$ prime congruent to $3$ modulo $4$, then subgroups of order $4n$ do exist in $S_n$ (Proposition \ref{pr11} in the Appendix), but we do not know
if any such subgroup can be packed into a Latin square.

\section{Mutually orthogonal Latin squares}

A pair of Latin squares $L$ and $L'$ of order $n$ are  \textit{orthogonal} if the ordered pairs $(L(i,j),L'(i,j))$ are distinct for all $i, j\in \{1,\ldots,n\}$. 
A set of Latin squares is called \textit{mutually orthogonal} (MOLS) if each Latin square in the set is orthogonal to every other Latin square of the set.

\begin{theorem}\label{orto}
For every prime number $p$ congruent to $1$ modulo $4$, there is a permutation group $G$ of order $(p-1)p$ and a set of $(p-1)/4$ mutually orthogonal Latin squares of order $p$ such that every permutation in $G$ occurs exactly once as a line of  one of these Latin squares. 
\end{theorem}

\begin{proof} (Method based on Bose \cite{bose})
Let $p$ be a prime number congruent to $1$ mod $4$, with the arithmetic of $GF(p)$ on the set $\{1,\ldots,p\}$. Using the subset $V= \{1, \ldots,(p-1)/2\}$ let us form $(p-1)/4$ ordered couples $(r,s)$ so that every member of $V$ occurs exactly once as a (first or second) component of such a couples $(r,s)$. It is easy to see that the couples can be formed in a way that for any two distinct couples $(r,s)$ and $(r',s')$ the determinant of the matrix 
$$\begin{pmatrix}
r & s \\
r' & s' 
\end{pmatrix}$$
is non-zero. For each of these couples $(r,s)$ define the $p$-by-$p$ matrix $M_{rs}$ by $$M_{rs}(i,j)= ri + sj$$ The $(p-1)/4$ matrices so defined are mutually orthogonal. As the set $V$ does not contain the negative of any of its members, every permutation of the elements of $GF(p)$ given by the linear permutation polynomial $f(x)=rx+c$, with $r \in V$ and $c \in GF(p)$, will occur exactly once as a row or column of one of the $(p-1)/4$ matrices given, and every permutation of the elements of $GF(p)$ given by $f(x)=rx+c$, with $r \in GF(p)-\{0\}$ and $c \in GF(p)$ will occur exactly once as a line of one of these matrices. Moreover, the $(p-1)p$ permutations having this linear (affine) form constitute a group.
\end{proof}

\section{Minimum number of lines and symmetries of Latin squares}

The question that is opposite to the one asking for as many distinct lines as possible in Latin squares is that of how many of the $4n$ lines of a Latin square of  order $n$ must be distinct. More precisely, let  $min(n)$ be the first positive integer $m$  such that every Latin square of order $n$ has at least  $m$ distinct lines, and let $min(L)$ denote the number of distinct lines in a given Latin square $L$.  By the very definition of Latin squares,  $min(n)$ is obviously at least $n$.  In this section we shall see that the value of $min(n)$ depends on the parity of $n$, and for a given Latin square $L$ of order $n$, it depends on the symmetry properties of $L$ whether $min(L)\geq  min(n)$.

\begin{theorem}
 If $n>1$ is odd, then every Latin square of order $n$ has at least $2n$ distinct lines, and a Latin square of order $n$ having $2n$ distinct lines exists. For even $n$, every Latin square of order $n$ has at least $n$ distinct lines, and a Latin square of order $n$ having $n$ distinct lines exists.
\end{theorem}

\begin{proof} If a Latin square of odd order $n>1$ had less than $2n$ distinct lines, then some line would have to appear at least three times in $L$. It would then have to appear either both as a row and a reverse row, or both as a column and a reverse column. In  the first case the middle elements of two rows would have to be the same, which is impossible in a Latin square, and the second case is ruled out similarly. On the other hand, given any $n>1$ odd number, the matrix $M(i,j)=i+j-1$ (with addition $\mod n$) is a Latin square with exactly $2n$ lines (see Proposition \ref{hus}).

For the even case, we only need to show that for any even $n = 2m$ a Latin square of order $n$ having only $n$ distinct lines exists. Indeed, Let $A$ be a symmetric Latin square of order $m$.  Let another $m$-by-$m$ matrix $B$ be defined by $B(i,j)=A(i,j)+m$.  Then, with $J$ denoting the  $0-1$ matrix with $1$'s only on the Hankel diagonal (reversal matrix), the following matrix is a Latin square of order $n$ having only $n$ distinct lines:

\[\left[\begin{array}{c|c}
A&BJ\\ \hline
JB&JAJ\end{array}\right]\]

\noindent Note that this $n \times n$  matrix is symmetric and also centrosymmetric. The construction is similar to a construction for centrosymmetric Latin squares 
 given in \cite{cela}.
\end{proof}
\begin{theorem}
A Latin square $L$ of odd order $n>1$ has exactly $2n$ distinct lines if and only if $L$ is symmetric or Hankel symmetric. A Latin square $L$ of even order $n>0$ has exactly $n$ distinct lines if and only if $L$ is both symmetric and Hankel symmetric.
\end{theorem}

\begin{proof}
\textit{For odd} $n=2m-1$: if $L$ is symmetric, then each of the $n$ rows  also appears as a column. Therefore this same line cannot additionally also appear as a reverse row or a reverse column, because this would result in a coincidence of middle elements, which is impossible in a Latin square. It follows that reverse rows, also appearing by symmetry as reverse columns, make up an other set of $n$ disjoint lines. Similarly, if $L$ is Hankel symmetric, then each of the $n$ rows  appears as a reverse column and each of the $n$ reverse rows appears as a column resulting in a total of $2n$ distinct lines.

To verify the converse, suppose that  $L$ has exactly $2n$ distinct lines. No column can appear as a reverse column. Thus the set of columns and reverse columns must contain $2n$ distinct lines. Each row must therefore appear in this set, i.e. each row must also appear either as a column or as a reverse column. 

In particular the first row of $L$ would have to appear either as a column or as a reverse column, and indeed either as the first column or as the last reverse column of the square. 

\smallskip

\textit{Case 1}: the first row of $L$ appears as the first column.  It is easy to see that this assumption implies that the middle row has to coincide with the middle column. 

We now claim that $L$ is symmetric.

Suppose the contrary, that $L$ is not symmetric, and let the $i$-th row be distinct from the $i$-th column. This will lead to a contradiction.

 Observe that the $i$-th row must coincide either with a column or with a reverse column. 

\emph{Subcase 1}: if the $i$-th row coincides with the $j$-th column, the first elements of these two lines being the same number $x$, $i$ would have to be equal to $j$ because $x$ appears in the same place of the first row and also of the first column, and thus the $i$-th row would be after all the same as the $i$-th column.

\emph{Subcase 2}: if the $i$-th row coincides with the $j$-th reverse column, then $i$ cannot be $j$ because in that case the first and last elements of the $i$-th row would be the same, which is impossible. The relationship between $i$ and $j$ must be $i+j = n+1=2m$ (see Figure \ref{abra}). Thus $i,j$ and $m$ are three distinct numbers and $m$ is between $i$ and $j$.
Let $x$ denote the middle element of the $i$-th row, this is also the $i$-th element of the middle column.
Clearly $x$ is also the middle element of the $j$-th reverse column, in other words $x$ is also the $j$-th element of the middle row.
But then $x$ is also the $j$-th element of the middle column (because the middle row and the middle column are the same). It is thus both the $i$-th and the $j$-th element of the middle column: a contradiction showing that $L$ is indeed symmetric as claimed.

\smallskip

\textit{Case 2}:  the first row of $L$ appears as the last reverse column. A similar reasoning verifies that this assumption  leads to the conclusion that $L$ is Hankel symmetric.

\begin{figure}[ht!]
   \includegraphics[width=0.69\linewidth]{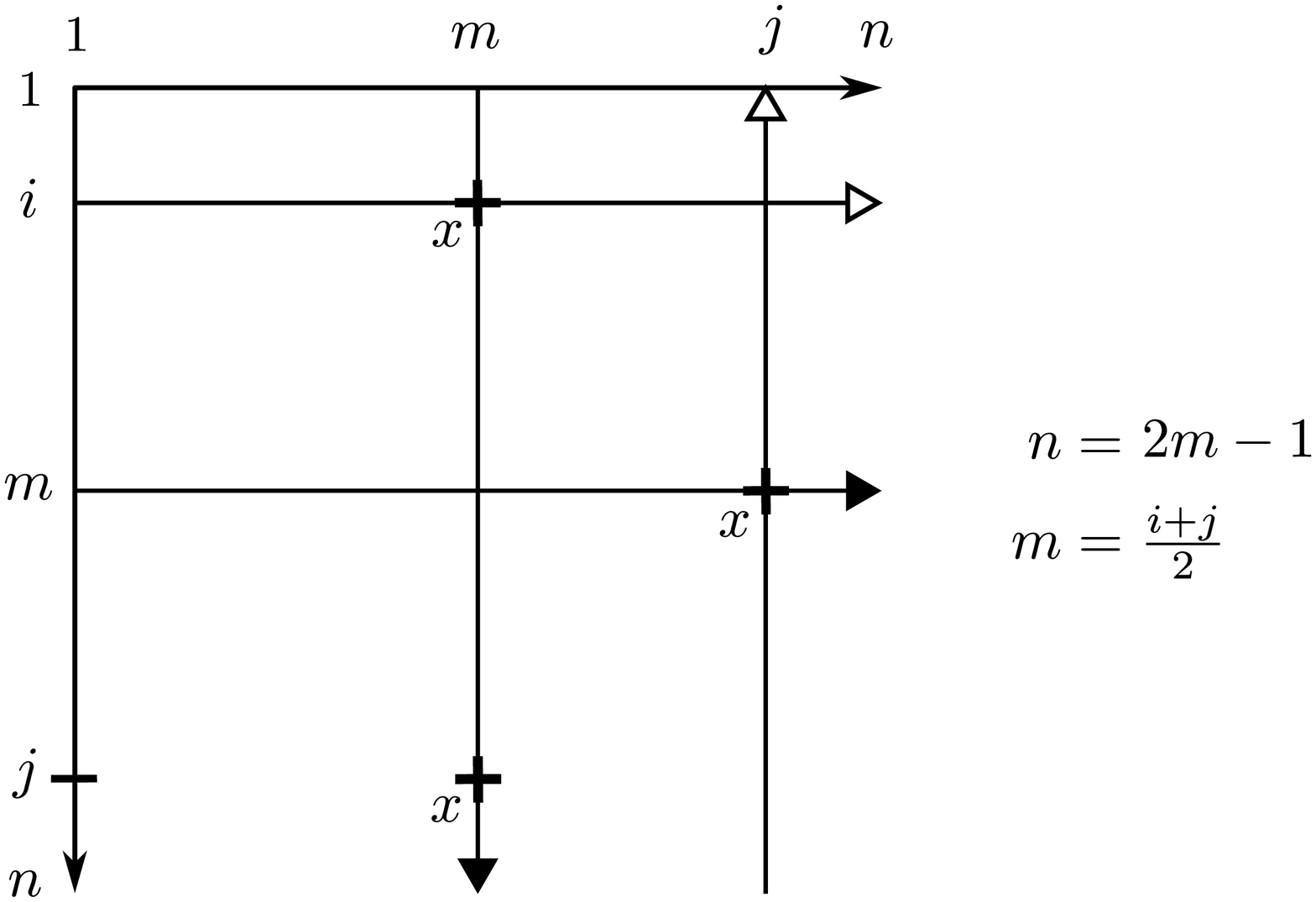}
   \caption{}
   \label{abra}
\end{figure}
\smallskip
\textit{For even} $n$:  if $L$ is both symmetric and Hankel symmetric, then it is also centrosymmetric, and every row must appear also as a column, as a reverse row and as a reverse column as well. There can be thus no more lines appearing in $L$ than the $n$ rows.

To see the converse, suppose that $L$ has only $n$ distinct lines. Each line must therefore appear $4$ times: it must appear as a row, as a column, as a reverse row and as a reverse column as well. But the $i$-th row can only appear as the $i$-th column, and thus $L$ must be symmetric. For a similar reason, the $j$-th reverse row must appear as the reverse column that is the $(n-j+1)$-th from the left, and thus $L$ is also Hankel symmetric.
\end{proof}
\section*{Appendix: Existence of permutation groups of order $4n$}

\begin{proposition}\label{pr10}
There is no order $4n$ subgroup in $S_n$ for prime $n$ congruent to $3$ modulo $4$.
\end{proposition}
\begin{proof}
Suppose the contrary, let $G$ be such a subgroup of $S_n$ of order $4n$.
By Sylow, $G$ has a subgroup $C$ of order $n$. Choose any generator $g$ of $C$, it must have a unique cycle of length $n$. 
Also by Sylow, some permutation $r$ of order $2$ (involution) must also belong to $G$ 
and   without loss of generality   to the stabilizer of $(n+1)/2$ in $G$.
The stabilizer subgroups of $G$ are all conjugates of each other by elements of $C$, and they all have the same order.
Also, under the conjugation establishing an isomorphism between two stabilizers, corresponding elements of the stabilizers have similar cycle structure, in particular they have the same number of fixed points.

As $G$ is transitive with orbit of size $n$, the stabilizers must have order $4$.
We use the fact that a $4$-element group is either cyclic or it is isomorphic to the direct product $\mathbb{Z}_2\times \mathbb{Z}_2$ (Klein group).

\emph{Case 1}: the stabilizer $F$  of $(n+1)/2$  is cyclic.
Then it is generated by a permutation $f$ in $F$ the square of which is $r$.
Clearly $f$ must have at least one cycle of length $4$, but since $n$ is congruent to $3$ mod $4$, there must be 
at least $2$ elements of $\{1,\ldots,n\}$ different from $(n+1)/2$ that are not on any cycle of length $4$. 
Such an element $x$ is either a fixed point of $f$ or it belongs to a cycle of length $2$.
In either case the square of $f$ must fix every such element $x$. Thus $r$ has at least $3$ fixed points. 

Applying the Burnside Lemma to the group $G$ acting on $\{1,\ldots,n\}$ (permutations being functions from $\{1,\ldots,n\}$ to itself), the sum of the number of fixed points of the various permutations $h \in G$ taken over all members $h$ of $G$ should be $4n$.
The identity permutation has $n$ fixed points.
Each of the $n$ conjugates of $r$ by members of $C$ has at least $3$ fixed points, these are at least $3n$ points altogether. 
Each of the $n$ conjugates of $f$ by members of $C$ has at least the fixed point  $(n+1)/2$, the sum of these is at least $n$.  

The $n$ conjugates are distinct since they are not all the same, and they form an orbit under conjugation by the elements of the prime subgroup $C$.
Thus the Burnside count of fixed points is at least $n+3n+n=5n$: we have obtained a contradiction.

\emph{Case 2}: The stabilizer $F$ of $(n+1)/2$ is a Klein group.
Then it consists of the identity permutation, the permutation $r$, and two other permutations $s$ and $q$ of order $2$, such that the product of any non-coinciding two of  $r,s,q$ is the third.
The cycle structure of each of $s$ and $q$ consists of at least one cycle of length $2$ and an   odd  number of fixed points  including the point $(n+1)/2$.

Consider the complete graph $K$ on the $(n-1)$-element vertex set $$V=  \{1,\ldots,n\} - \{(n+1)/2\}$$
The permutations $r,s,q$ define respective matchings $R,S,Q$ in $K$: a pair of vertices is an edge in the matching $R,S,Q$ if it is transposed by $r,s,q,$ respectively.
Without loss of generality assume that the size of the matching $R$ is not less than the size of $S$ or $Q$.
The connected components of the union of any two matchings can be only circuits of even length and paths. 

We claim that any such circuit must have length $4$.
For suppose that one of these circuits, say in the union of matchings $M$ and $N$, corresponding to two of the three permutations $r,s,q,$ say to the permutations $f,g,$ is a circuit of length $2m$, $m>2$. This would imply that the vertices of this circuit would not be fixed by the composite permutation $fgfg$, thus $fg$ would not be of order $2$, as any non-identity element of a Klein group should be: a contradiction proving the claim that all circuits in the union of any two of $R,S,Q$ have length $4$. 

Next, observe that if $(u,v,w,x)$ is a circuit (of length $4$) in the union of $S$ and $Q$, then, due to $sq=qs=r$, the edges $uw$ and $vx$ belong to $R$.

We claim that it is not possible for each of $S$ and $Q$ to cover all vertices in $V$. For in that case, by the maximality assumption on $R$, the matching $R$ would also cover all vertices in $V$.  As $n-1$ is congruent to $2$ mod $4$, there would be an odd number $m$ of edges in $R$ that are not incident with any circuit in the union of $S$ and call these edges free. 
As $r=sq$, the end vertices of each free edge must be connected by a unique path of length  $2$ in the union of $S$ and $Q$.

We show that the end vertices of a free edge $xy$ in $R$ cannot indeed be connected by a path of length $2$ in the union of $S$ and $Q$.
For suppose that they are so connected by the path $xvy$, with $xv$ is $S$ and $vy$ in $Q$.
Let $w$ be the vertex $r(v)$. Using $rq=s$ and $rs=q$, it is easy to verify that $(x,v,y,w)$ would be a circuit of length $4$ in the union of $S$ and $Q$, contradicting the assumption that $xy$ is a free edge and proving the claim.

In other words, every free edge of $R$ is either in $S$ or $Q$. (A free edge $xy$ cannot be in both $S$ and $Q$, because then we would have $sq(x)=x$, contradicting $r(x)=y$).
Without loss of generality at least one free edge $xy$ of $R$ is not in $S$. 

The set $fix(s)$ of fixed points of $s$ consists of $(n+1)/2$ and the end vertices of the free edges of $R$ that are not in $S$. Clearly $fix(s)$ has at least $3$ elements, say $k$ elements. 
Obviously $k<n$.
The group $C$ acts on $G$ by conjugation, under this action the orbit of $s$ is not trivial, and as its size must be a divisor of the order $n$ of $C$, which is prime, this orbit must contain $n$ distinct permutations, all having $k$ fixed points, where $1<k<n$.

Applying the Burnside Lemma to the group $G$ acting on $\{1,\ldots,n\}$ (permutations being functions from $\{1,\ldots,n\}$ to itself), the sum of the number of fixed points of the various permutations $f \in G$ taken over all members $f$ of $G$ should be $4n$.
The indentity permutation has $n$ fixed points.
Each of the $n$ conjugates of $r$ by members of $C$ has $1$ fixed points, these are $n$ points altogether.
Each of the $n$ conjugates of $s$ by members of $C$ has $k$ fixed points, the sum of these is $kn$.
As $n+n+kn$ is at least $5n$, we have obtained a contradiction completing the proof.
\end{proof}

\begin{proposition}\label{pr11}
For even integers $n=2m\geq 6$  the full symmetric group $S_n$ always has a subgroup of order $4n$.
\end{proposition}

\begin{proof}
For any integer $m\geq4$  let us consider the subgroup consisting of permutations $f$ defined on the set $\{1,\ldots,m,m+1, \ldots,2m \}$ for which 
\begin{enumerate}[label=(\roman*)]

\item  each of the sets $\{1,\ldots,m\}$, $\{m+1,m+2\}$ and $\{m+3,m+4\}$ is invariant under the permutation $f$,
\item every $i\in \{m+5, \ldots, 2m \}$ is a fixed point of $f$,
\item  there exists an $i$ and a $k\in \{-1,1\}$ such that for every $j\in \{1,\ldots,m\}$,  $f(j)$ is congruent to $i+kj$ modulo $m$.
\end{enumerate}
The number of permutations in this subgroup ---   isomorphic to the direct product $D_m\times \mathbb{Z}_2^2$ of a dihedral group and the Klein group  ---  is indeed $2^3m=4n$.

For the case $m=3$ let us consider the subgroup of order $4n=24$ consisting of permutations $f$ defined on the set $\{1,\ldots,6\}$ for which 
 $5$ and $6$ are fixed points under $f$.
\end{proof}

\end{document}